\documentclass[11pt,reqno]{amsart}
\usepackage{amsfonts}
 \usepackage{amsmath}
 \usepackage{amssymb}

\theoremstyle{plain}

\newtheorem{thm}{Theorem}

\newtheorem{prop}{Proposition}

\newtheorem{lem}{Lemma}

\newtheorem{rem}{Remark}

\newtheorem{defn}{Definition}

\theoremstyle{plain}

\title[Nonlinear wave equation]
{On the nonlinear wave equation with time periodic potential}

\author[V. Petkov]{Vesselin Petkov}
\address{Institut de Math\'ematiques de Bordeaux, 351,
Cours de la Lib\'eration, 33405  Talence, France}
\email{petkov@math.u-bordeaux.fr}
\author[N. Tzvetkov]{Nikolay Tzvetkov}
\address{D\'epartement de Math\'ematiques (AGM ), Universit\'e de Cergy-Pontoise, 2, av. Adolphe Chauvin, 95302 Cergy-Pontoise Cedex, France}
\email{nikolay.tzvetkov@u-cergy.fr}

\begin{document}

\numberwithin{equation}{section}

\maketitle

\def\R{{\mathbb R}}
\def\C{{\mathbb C}}
\def\N{{\mathbb N}}

\def\diam{\mbox{\rm diam}}
\def\rr{{\cal R}}

\def\e{\emptyset}
\def\dQ{\partial Q}
\def\dk{\partial K}
\def\endofproof{{\rule{6pt}{6pt}}}
\def\di{\displaystyle}
\def\dist{\mbox{\rm dist}}

\def\u-{\overline{u}}
\def\du{\frac{\partial}{\partial u}}
\def\dv{\frac{\partial}{\partial v}}
\def\dt{\frac{d}{d t}}
\def\dx{\frac{\partial}{\partial x}}
\def\con{\mbox{\rm const }}
\def\Box{\spadesuit}
\def\nn{{\cal N}}
\def\mm{{\cal M}}
\def\kk{{\cal K}}
\def\ll{{\cal L}}
\def\vv{{\cal V}}

\def\ff{{\cal F}}
\def\i{{\bf i}}

\def\curl{{\rm curl}\,}
\def\dive{{\rm div}\,}
\def\grad{{\rm grad}\,}

\def\dist{\mbox{\rm dist}}
\def\pr{\mbox{\rm pr}}
\def\pp{{\cal P}}
\def\supp{\mbox{\rm supp}}
\def\Arg{\mbox{\rm Arg}}
\def\In{\mbox{\rm Int}}
\def\Re{\mbox{\rm Re}}
\def\li{\mbox{\rm li}} 
\def\ep{\epsilon}
\def\tr{\tilde{R}}
\def\be{\begin{equation}}
\def\ee{\end{equation}}
\def\eps{\epsilon}


\def\ii{{\imath }}
\def\jj{{\jmath }}
\def\tc{\tilde{C}}
\def\II{{\cal I}}
\def\ccij{ \cc_{i'_0,j'_0}[\eta]}
\def\hc{\hat{\cc}}
\def\dd{{\cal D}}

\def\bs{\bigskip}
\def\hrho{\hat{\rho}}
\def\ms{\medskip}
\def\nn{\mathcal N}
\def\J{bf J}
\def\hC{\widehat{C}}
\def\lip{\mbox{\footnotesize\rm Lip}}
\def\Lip{\mbox{\rm Lip}}
\def\cn{{\mathcal N}}
\def\sn{{\mathbb  S}^{n-1}}
\def\Ker {{\rm Ker}\:}

\def\el{E_{\lambda}}
\def\Rc{{\mathcal R}}
\def\Ha{H_0^{ac}}
\def\la{\langle}
\def\ra{\rangle}
\def\Ko{\Ker G_0}
\def\Kd{\Ker G_d}
\def\uc{{\mathcal U}}
\def\fc{{\mathcal F}}
\def\hc{{\mathcal H}}

\begin{abstract} It is known that for some time periodic potentials $q(t, x) \geq 0$ having compact support with respect to $x$ some solutions of the Cauchy problem for the wave equation $\partial_t^2 u - \Delta_x u + q(t,x)u = 0$ have exponentially increasing energy as $t \to \infty$. We show that if one adds a nonlinear defocusing interaction $|u|^ru, 2\leq r < 4,$ then the solution of the nonlinear wave equation exists for all $t \in \R$ and its energy is polynomially bounded as $t \to \infty$ for every choice of $q$. Moreover, we prove that the zero solution of the nonlinear wave equation is instable if the corresponding linear equation has the property mentioned above.
\end{abstract}

\footnote{MSC [2010]: Primary 35L71, Secondary 35L15}

\section{Introduction}
Our goal in this paper is to show that a defocusing nonlinear interaction may improve, in a certain sense, the long time properties of the solutions of the wave equation with a time periodic potential.

Consider the Cauchy problem for the following potential perturbation of the classical wave equation in the Euclidean space $\R^3$
\begin{equation}\label{eq:1.1}
\partial_t^2 u-\Delta_x u+q(t,x)u=0,\quad u(0,x)=f_1(x),\,\, \partial_t u(0,x)= f_2(x)\,,
\end{equation}
where $0 \leq q(t, x) \in C^{\infty}(\R\times\R^3)$ is periodic in time $t$ with period $T > 0$ and has a compact support with respect to $x$ included in $\{x \in \R^3: |x| \leq \rho\},$
for some positive $\rho$. 
It is easy to show that the Cauchy problem \eqref{eq:1.1} is globally well-posed in  $\hc = H^1(\R^3) \times L^2(\R^3)$. The analysis of the long time behavior of the solution of 
 \eqref{eq:1.1} may be quite intricate (see e.g. \cite{P,CPR}).  A slight adaptation of the arguments presented in \cite{CPR} leads the following result.  
 \begin{thm}\label{th1}
 There exist $q$ and $(f_1,f_2)\in \hc$ such that the solution of \eqref{eq:1.1} satisfies :
 \begin{equation}\label{eq:1.2}
\exists\, C>0,\,\,\exists\, \alpha>0\quad {\rm such\,\, that}\quad \forall\, t\geq 0, \quad  \|u(t,\cdot)\|_{H^1(\R^3)}\geq C\, e^{\alpha t}\,. 
 \end{equation}
 \end{thm}
The above result has been established in \cite{CPR} for the Cauchy problem with initial data in $H = H_D(\R^3) \times L^2(\R^3)$.
In fact we show that the propagator of (\ref{eq:1.1})
$$V(T, 0): \hc \ni (f_1(x), f_2(x)) \longrightarrow (u(T, x), u_t(T, x)) \in \hc$$
has an eigenvalue $y, |y| > 1$ which implies (\ref{eq:1.2}).\\

 Our purpose is to show that adding a nonlinear perturbation to \eqref{eq:1.1} forbids the existence of solutions satisfying \eqref{eq:1.2}.
 Consider therefore the following Cauchy problem 
 \begin{equation}\label{eq:1.3}
\partial_t^2 u-\Delta_x u+q(t,x)u+|u|^r u=0,\quad u(0,x)=f_1(x),\,\, \partial_t u(0,x)= f_2(x)\,,
\end{equation}
where $2\leq r<4$. We have the following statement.  
 \begin{thm}\label{th2}
 For any choice of $q$ the Cauchy problem \eqref{eq:1.3} is globally well-posed in $\hc$. Moreover, for every 
 $(f_1, f_2)\in \hc$ there exists a constant $C > 0$ such that for every $t\in\R$, the solution of \eqref{eq:1.3} satisfies the polynomial bound 
 \begin{eqnarray*}
 \|\nabla u(t,\cdot)\|_{L^2(\R^3)}
 +
 \| \partial_t u(t,\cdot)\|_{L^2(\R^3)}
 \leq 2\Bigl(X(0)^{\frac{r}{r+2}}+C|t|\Bigr)^{\frac{r+2}{2r}},\\
\| u(t,\cdot)\|_{L^2(\R^3)}\leq \|f_1\|_{L^2} + 2|t|\Bigl(X(0)^{\frac{r}{r+2}}+C|t|\Bigr)^{\frac{r+2}{2r}}\,,
 \end{eqnarray*}
where
$$X(t)=\int_{\R^3} \big(\frac{1}{2}|\partial_t u|^2 + \frac{1}{2}|\nabla_x u|^2 + \frac{1}{2}q |u|^2 +\frac{1}{r+2} |u|^{r+2}\big)dx$$ 
and $C > 0$ depends only on $q$ and $r$.
 \end{thm}

By global well-posedness we mean the existence, the uniqueness and the continuous dependence with respect to the data. 
The proof of Theorem 2 is based on the equality
\begin{equation} \label{eq:1.4}
 X'(t) = \frac{1}{2}\Re \int_{\R^3} (\partial_t q) |u|^2 dx
\end{equation}
and the estimate
$$|X'(t)| \leq C X^{1 - \frac{r}{r+2}}(t).$$

It is classical to expect that
the result of Theorem~\ref{th1} implies the instability of the zero solution of \eqref{eq:1.3}. More precisely, we have the following instability result. 
\begin{thm}\label{th3}
With $q$ as in Theorem~\ref{th1} the following holds true. There is $\eta>0$ such that for every $\delta>0$ there exists $(f_1,f_2)\in \hc$ , $\|(f_1,f_2)\|_{\hc}<\delta$ and there exists $n = n(\delta)>0$ 
such that the solution of \eqref{eq:1.3} satisfies 
$
\|(u(nT,\cdot),\partial_t u(nT,\cdot)\|_{\hc}>\eta.
$
\end{thm}
We are not aware of any nontrivial choice of $(f_1, f_2)\in \hc$ such that the solution $u(t, x)$ of \eqref{eq:1.3} and $u_t(t, x)$ remain uniformly bounded in $\hc$ for all $t \geq 0.$ 
The paper is organized as follows. In the next section, we prove Theorem~\ref{th1}. The third section is devoted to the proof of Theorem~\ref{th2}. First we obtain a local existence and uniqueness result on intervals $[s, s + \tau]$ with $\tau = c(1 + \|(f_1, f_2)\|_{\hc})^{-\gamma}$ with constants $c > 0$ and $\gamma > 0$ independent on $f$. Next we establish (\ref{eq:1.4}) for solutions 
$$u(t, x) \in  C([0, A], H^{2}(\R^3)) \cap C^1([0, A], H^1(\R^3)) \cap L_t^{\frac{2r +2}{r-2}}([0, A], L^{2r +2}_x(\R^3))$$
and finally, by a local approximation in small intervals we justify (\ref{eq:1.4}) for every fixed $A > 0$ and $0 \leq t \leq A$.
In the fourth section, we prove Theorem~\ref{th3} passing to a system
$$w_{n+1} = \fc(w_n),\: n \geq 0,$$
where $\fc = \uc(0, T)$ is the propagator of the nonlinear equation. In the fifth section we discuss the generalizations concerning the nonlinear equations
$$\partial_t^2 u - \Delta_x u + |u|^r u +\sum_{j = 0}^{r-1} q_j(t, x)|u|^j u = 0, \:r = 2,3$$
with time-periodic functions $q_j(t + T_j, x) = q_j(t, x) \geq 0,\: j = 0,1,r-1$ having compact support with respect to $x$.
\section{Proof of Theorem~\ref{th1}}
\subsection{The linear wave equation with time periodic potential}
Let $H_D(\R^3)$  be the closure of $C_0^{\infty}(\R^3)$ with respect to the norm
$\|f\|_{H_D} = \|\nabla_x f\|_{L^2(\R^3)}$. Define the (energy) space 
$$
H = H_D(\R^3) \times L^2(\R^3)
$$ 
with norm
$$
\|f\|_0 =\big(\|f_1\|^2_{H_D} + \|f_2\|^2_{L^2}\big)^{1/2},\quad f = (f_1,f_2).
$$
Let $u(t, x; s)$ be the solution of the Cauchy problem
\begin{equation} \label{eq:2.1}
 \partial_t^2 u - \Delta_x u + q(t, x) u= 0,\: u(s, x) = f_1(x), \: \partial_t u(s, x) = f_2(x)
\end{equation}
 with $f = (f_1, f_2) \in H$.  Therefore
the operator
$$
H \ni f \rightarrow U(t, s)f = (u(t, x; s),\partial_t u(t, x; s)) \in H
$$
is called the propagator (monodromy operator) of (\ref{eq:2.1}) and there exist $C > 0$ and $\alpha \geq 0$ so that
\begin{equation} \label{eq:2.2}
\|U(t, s) f\|_{0} \leq C e^{\alpha|t- s|}\|f\|_{0}.
\end{equation}
Let $U_0(t-s)f = (u_0(t, x; s), \partial_t u_0(t, x; s))$, where $u_0$  solves $ \partial_t^2 u_0 - \Delta_x u_0 = 0$ with initial data $f$ for $t = s.$ Then we have 
\begin{equation} \label{eq:2.3}
U(t, s)f - U_0(t- s)f= - \int_s^t U_0(t- \tau)Q(\tau) U(\tau, s)f d\tau, 
\end{equation}
where 
$$
U_0(t)=
\begin{pmatrix}  \cos(t\sqrt{-\Delta}) &  \frac{\sin(t\sqrt{-\Delta})}{\sqrt{-\Delta}} \\
- \sqrt{-\Delta}  \sin(t\sqrt{-\Delta})  & \cos(t\sqrt{-\Delta}) \end{pmatrix},
\quad
Q(t) = \begin{pmatrix} 0 & 0\\
q(t, x) & 0\end{pmatrix}.$$
Using the relation \eqref{eq:2.3} and the compact support of $q$ allows us to obtain the estimate  
$$
\|U(t, s)f - U_0(t- s)f\|_{H^2(\R^3) \times H^1(\R^3)}\leq C \|U(t, s) f\|_{0}\,.
$$
Moreover the support property of $q$ also yields 
$${\rm supp}_x\: (
U(t, s)f - U_0(t- s)f
) \subset \{|x| \leq \rho +|t-s|\}\,.
$$
Consequently $U(t,s)$ is a {\it compact perturbation} of the unitary operator $U_0(t-s).$

Now consider the space $\hc = H^1(\R^3) \times L^2(\R^3) \subset H$ with norm
$$\|f\|_1 = \big (\|f_1\|^2_{H^1(\R^3)} + \|f_2\|^2_{L^2(\R^3)}\big)^{1/2},\quad
\|f_1\|^2_{H^1(\R^3)} = \|\nabla_x f_1\|^2_{L^2(\R^3)} + \|f_1\|^2_{L^2(\R^3)}.$$
The map $U_0(t)$ {\it  is not unitary} in $\hc.$  However, one easily checks that

$$\|U_0(t) f\|_1 \leq C(1+|t|) \|f\|_1,\quad \forall \, t \in \R,$$ 
with a constant $C > 0$ independent of $t$.
Consequently, the spectral radius of the operator $U_0(T): \hc \to \hc$ is not greater than 1.

By using (\ref{eq:2.3}), it is easy to show by a fixed point theorem that for small $t_0 >0$ and $s \leq t \leq s +t_0$ we have a local solution 
$
(v(t, x; s) , \partial_t v(t, x; s)) \in \hc
$
 of the Cauchy problem (\ref{eq:2.1}) with initial data $f \in \hc.$ For this solution
 one deduces
$$\frac{d}{dt}\int_{\R^3} 
 \big( |\partial_t v(t, x; s)|^2 + |\nabla_x v(t, x; s)|^2 + |v(t, x; s)|^2\big)dx   = - 2 \Re\int_{\R^3} q v\overline{\partial_t v} dx +2\Re\int_{\R^3} v\overline{\partial_t v}dx $$
which yields
$$\frac{d}{dt} \|(v(t,x;s), \partial_t v(t, x; s))\|^2_1 \leq C_1 \|(v(t,x; s), \partial_t v(t, x; s))\|^2_1$$
with a constant $C_1 > 0$ independent of $f$ and $s$. 
The last inequality implies an estimate
$$\|(v(t,x; s), \partial_t v(t, x; s))\|_1 \leq C_2 e^{\beta |t- s|} \|f\|_1,\quad s \leq t \leq s +t_0,\, \beta \geq 0.$$
By a standard argument this leads to a global existence of a solution of (\ref{eq:2.1}).
Introduce the propagator 
$$\hc \ni f \mapsto V(t, s)f = (v(t, x; s),\partial_t v(t, x; s)) \in \hc$$
corresponding to the Cauchy problem (\ref{eq:1.1}) with initial data $f \in \hc.$ As in Section 5 in \cite{P} it is easy to see that
we have the following properties
$$
U(t, s) \circ U(s, r) = U(t, r),\,\,  U(s, s) = {\rm Id},\,\, U(t + T, s + T) = U(t, s),\quad  t, s, r \in \R.
$$
The same properties hold for the propagator $V(t, s)$.  In particular, $V(T, 0) = V((k+1)T, kT)$,  $k \in \N$ and $V(nT, 0) = (V(T, 0))^n.$

As above notice  that $V(t, s) - U_0(t-s)$ is a compact operator in ${\mathcal L}(\hc)$. For $|z| \gg 1$ we have
$$(V(T, 0) - z I)^{-1} = (U_0(T) - z I)^{-1} - (U_0(T) - zI)^{-1}\big(V(T, 0) - U_0(T)\big) (V(T, 0) - zI)^{-1}$$ 
hence
$$ \big[ I + (U_0(T) - z I)^{-1} \big(V(T, 0) - U_0(T)\big)\big](V(T, 0) - zI)^{-1} = (U_0(T) - z I)^{-1}.$$
Set $K(z) = I + (U_0(T) - z I)^{-1} \big(V(T, 0) - U_0(T)\big)$. For $|z|$ large enough $K(z)$ is invertible. 
By the analytic Fredholm theorem for $|z| \geq 1 + \delta > 1$  the operator $K(z)$ is invertible outside a
discreet set and the inverse $K(z)^{-1}$ is a meromorphic operator-valued function. Consequently, the operator $V(T, 0) \in {\mathcal L}(\hc)$ has in the open domain $|z| > 1$ a discreet set of eigenvalues with finite multiplicities which could accumulate only to  the circle $|z| = 1$.
\subsection{Extending the result of \cite{CPR} to $\hc$}
In \cite{CPR} it was proved that there are potentials $q(t, x) \geq 0$ for which the operator $U(T, 0): H \to H$ has an eigenvalue $z, |z| > 1$. In this paper we deal with the operator $V(T, 0): \hc \to \hc$ and  it is not clear if the eigenfunction $\psi \in H$ with eigenvalues $z$ constructed in \cite{CPR} belongs to $\hc.$

Below we make some modifications on the argument of \cite{CPR} in order to show that for  the potential constructed in \cite{CPR} the corresponding operator $V(T, 0): \hc \to \hc$ has an eigenvalue $y, |y| > 1.$ For convenience we will use the notations in \cite{CPR} and we recall some of them. The potential in \cite{CPR} has the form
$V^{\eps}(t, x) := b^{\eps}(x) + q(t)\chi^{\delta}(x)$
with $\eps > 0,$ where $b^{\eps}(x) \in C_0^{\infty}(\R^3)$ is
supported in $\{0 < L \leq |x| \leq L + 1\}$
and equal to $1/\eps$ for $\{L + \eps \leq |x| \leq L + 1 - \eps\}$, $\chi^{\delta}(x) \geq 0$ is a smooth function with support
in $|x| < L$ and equal to 1 for $|x| \leq L - \delta < L.$ 
Finally, $q(t) \geq 0$ is a periodic smooth function with period $T > 0$. 
The number $L$ is related to the interval of instability of the Hill operator associated with $q(t)$.
The number $\delta >0$ is fixed sufficiently small and the propagator $K^{\delta}(T)$ related to the equation
$$\partial_t^2 u - \Delta_x u + q(t)\chi^{\delta}(x) u = 0, \: t \geq 0,\: |x| < L$$
with Dirichlet boundary conditions on $|x| = L$ has an eigenvalue $z_1, |z_1| > 1$ with eigenfunction $\varphi \in H^1_0(|x| \leq L)$, that is $K^{\delta}(T)\varphi = z_1\varphi.$
Let $S^{\eps}(T): H \to H$ be the propagator corresponding to the Cauchy problem for the equation
$$\partial_t^2 u - \Delta_x u + V^{\eps}(t, x) u = 0,\: t \geq 0, \: x\in \R^3$$
and let $W^{\eps}(T): \hc \to \hc$ be the propagator for the same problem with initial data in $\hc.$ The problem is to show that
for $\eps > 0$ sufficiently small $W^{\eps}(T)$ has an eigenvalues $y, |y|>1$ (Here $S^{\eps}(T), W^{\eps}(T)$ correspond to our notations $U(T, 0), V(T, 0)$ and these operators have domains $H$ and $\hc$, respectively). \\

Extend $\varphi$ as 0 outside $|x| \geq L$ and denote the new function $\varphi \in \hc$ again by $\varphi.$ Let 
$$\gamma = \{z \in \C: |z - z_1| = \eta > 0\} \subset \{z: |z| > 1\}$$
be a circle with center $z_1$ such that $K^{\delta}(T) - zI$ is analytic on $\gamma$ and $z_1$ is the only eigenvalue of $K^{\delta}(T)$ in $|z- z_1| \leq \eta.$  If $W^{\eps}(T)$ has an eigenvalues on $\gamma$ the problem is solved. Assume that $W^{\eps}(T)$ has no eigenvalues on $\gamma$. It is easy to see that
$$(W^{\eps}(T) - zI)^{-1} \varphi = (S^{\eps}(T) - z I)^{-1}\varphi \in \hc,\: z \in \gamma.$$
Indeed,
$$(W^{\eps}(T) - zI)^{-1} \varphi = (S^{\eps}(T) - z I)^{-1}\varphi  + (S^{\eps}(T) - z I)^{-1}(S^{\eps}(T)- W^{\eps}(T))(W^{\eps}(T) - zI)^{-1} \varphi$$
and
$$(S^{\eps}(T)- W^{\eps}(T))(W^{\eps}(T) - zI)^{-1} \varphi = 0.$$
Our purpose is to study
$$(\varphi, (W^{\eps}(T) - zI)^{-1} \varphi)_{\hc} = (\varphi, (S^{\eps}(T) - zI)^{-1}\varphi)_{\hc},$$
where $(.,,)_{\hc}$ denotes the scalar product in $\hc$ and $(., .)_H$ denotes the scalar product in $H$.  It  was proved in \cite{CPR} that for $ z \in \gamma$ one has the weak convergence in $H$
$$(S^{\eps}(T) - zI)^{-1}\varphi \rightharpoonup_{\eps\to 0} (K^{\delta}(T) - zI)^{-1}\varphi,$$
so
$$(\varphi, (S^{\eps}(T) - zI)^{-1}\varphi)_H \longrightarrow (\varphi, (K^{\delta}(T) - zI)^{-1}\varphi)_H. $$
Here we have used the fact that $\varphi = 0$ for $|x| >L.$ Let $\varphi = (\varphi_1, \varphi_2).$
We claim that as $\eps \to 0$ we have
\begin{equation} \label{eq:2.4}
(\varphi_1, ((S^{\eps}(T) - zI)^{-1}\varphi)_1)_{L^2} \longrightarrow (\varphi_1, ((K^{\delta}(T) - zI)^{-1}\varphi)_1)_{L^2}.
\end{equation}
To prove this write
$$\varphi_1 =-\Delta \psi\: {\rm with}\:\psi =\Bigl(\frac{1}{4\pi |x|} \star \varphi_1\Bigr).$$
The main point is the following
\begin{lem} We have $\psi \in H_D(\R^3).$
\end{lem}
\begin{proof} Since
$$|\partial_{x_j} \psi(x)| = \Bigl| \frac{1}{4 \pi} \int_{\R^3} \frac{(x_j - y_j) \varphi_1(y)}{|x-y|^3}dy \Bigr| \leq  \frac{1}{4 \pi} \int_{\R^3} \frac{ |\varphi_1(y)|}{|x-y|^2} dy ,$$
we can apply the Hardy-Littlewood-Sobolev inequality. More precisely, by using Theorem~4.3 of \cite{LL} with $n = 3,\: \lambda = 2,\:r= 2,\: p= 6/5$, we obtain that
$$\|\partial_{x_j}\psi(x)\|_{L^2(\R^3)} \leq C \|\varphi_1(x)\|_{L^{6/5}(\R^3)}.$$
Now using that $\varphi_1(x)$  is with compact support and the H\"older inequality, we obtain that
$$\|\varphi_1(x)\|_{L^{6/5}(\R^3)} \leq C_1 \|\varphi_1(x)\|_{L^2(\R^3)}.$$
This completes the proof of Lemma 1. 
\end{proof}

Therefore
$$(-\Delta \psi, ((S^{\eps}(T) - zI)^{-1}\varphi)_1)_{L^2} = \Bigl(\langle\nabla_x \psi, \nabla_x ((S^{\eps}(T) - zI)^{-1}\varphi))_1\rangle\Bigr)_{L^2}$$
$$\longrightarrow_{\eps \to 0} \Bigl(\langle\nabla_x \psi, \nabla_x((K^{\delta}(T) - zI)^{-1}\varphi))_1\rangle\Bigr)_{L^2} = (-\Delta \psi, ((K^{\delta}(T) - zI)^{-1}\varphi))_1)_{L^2}$$
which proves the claim (2.4). Consequently,
\begin{equation} \label{eq:2.5}
(\varphi, (W^{\eps}(T) - zI)^{-1}\varphi)_{\hc} \longrightarrow (\varphi, (K^{\delta}(T) - zI)^{-1}\varphi)_{\hc}.
\end{equation}
Moreover, Proposition~4.2 in  \cite{CPR} says that with a constant $C_0 > 0$ we have uniformly for $z \in \gamma$ the norm estimate
$$\|(S^{\eps}(T) - zI)^{-1}\|_{H} \leq C_0,\: \forall \eps\in ]0,\eps_0].$$
Since 
$$\|(S^{\eps}(T) - zI)^{-1}\varphi\|_{L^2(|x| \leq L)} \leq C_1 \|(S^{\eps}(T) - zI)^{-1}\varphi\|_H,$$
the sequence $(\varphi, (W^{\eps}(T)- zI)^{-1} \varphi)_{\hc}$ is bounded for $ z \in \gamma$.
Repeating the argument of Section~5 in \cite{CPR}, one deduces
$$\Bigl(\varphi,\frac{1}{2\pi i}\int_{\gamma} (W^{\eps}(T) - zI)^{-1}\varphi dz\Bigr)_{\hc}\longrightarrow \Bigr(\varphi, \frac{1}{2 \pi i}\int_{\gamma}(K^{\delta}(T) - zI)^{-1}\varphi dz\Bigr)_{\hc} = \|\varphi\|^2_{\hc} \neq 0.$$
This completes the proof that  for small $\eps$ the operator $W^{\eps}(T)$ has an eigenvalue $y, |y| > 1.$
\section{Proof of Theorem~\ref{th2}}
\subsection{Local well-posedness}
Consider the linear problem 
\begin{equation} \label{eq:3.1}
 \partial_t^2 u - \Delta_x u + q(t, x) u= F,\: u(s, x) = f_1(x), \: \partial_t u(s, x) = f_2(x).
\end{equation}
By using the argument in  \cite{P}, one may show that the solution of \eqref{eq:3.1} satisfies the same {\it local in time}  Strichartz estimates as in the case $q=0$. Notice that for these local Strichartz estimates we don't need a global control of the
local energy and we can establish them without a condition on the cut-off resolvent $\varphi (V(T, 0) - z)^{-1}\varphi$.
More precisely for every finite $a > 0$ and $f = (f_1, f_2) \in \hc, F \in L^1([s,s+a];L^2(\R^3))$ we have that the solution of \eqref{eq:3.1} satisfies 
\begin{equation}\label{eq:3.2}
\|(u,\partial_t u)\|_{C([s,s+a];\hc)}+\|u\|_{L^p_t([s, s+a], L^q_x(\R^3))}\leq C(a)
\big(\|(f_1,f_2)\|_{\hc}+\|F\|_{L^1([s,s+a];L^2(\R^3))}\big)\,,
\end{equation}
provided $\frac{1}{p}+\frac{3}{q}=\frac{1}{2}$, $p>2$ (the constant $C(a)$ in \eqref{eq:3.2} depends on $a$, $p$ and $q(t,x)$). 
Moreover, if $(f_1, f_2) \in H^2(\R^3) \times H^1(\R^3)$ and $F \in L^1([s,s+a];H^1(\R^3))$, we have
\begin{multline}\label{eq:3.3}
\|(u,\partial_t u)\|_{C([s,s+a];H^2\times H^1)}+\|\nabla_x u\|_{L^p_t([s, s+a], L^q_x(\R^3))}
\\
\leq C(a)
\big(\|(f_1,f_2)\|_{H^2\times H^1}+\|F\|_{L^1([s,s+a];H^1(\R^3))}\big)\,.
\end{multline}
A standard application of \eqref{eq:3.2}, \eqref{eq:3.3} is the following local well-posedness result for the nonlinear wave equation
\begin{equation} \label{eq:3.4}
 \partial_t^2 u - \Delta_x u + q(t, x) u+|u|^r u= 0,\: u(s, x) = f_1(x), \: \partial_t u(s, x) = f_2(x),\quad 2\leq r<4.
\end{equation}
\begin{prop}\label{LWP}
There exist $C>0$, $c>0$ and $\gamma>0$ such that for every $(f_1,f_2)\in \hc$ there is a unique solution 
$(u, \partial_t u) \in C([s, s +\tau], H^1(\R^3) \times L^2(\R^3))$ of \eqref{eq:3.4} on $[s,s+\tau]$ with $\tau=c(1+\|(f_1,f_2)\|_{\hc})^{-\gamma}$. 
Moreover, the solution satisfies 
\begin{equation}\label{eq:3.5}
\|(u,\partial_t u)\|_{C([s,s+\tau];\hc)}+\|u\|_{L^{\frac{2r+2}{r-2}}_t([s, s+\tau], L^{2r+2}_x(\R^3))}\leq C \|(f_1,f_2)\|_{\hc}\,.
\end{equation}
If in addition $(f_1,f_2)\in H^2(\R^3)\times H^1(\R^3)$, then $(u,\partial_t u)\in C([s,s+\tau];H^2(\R^3)\times H^1(\R^3))$. 
\end{prop}
\begin{rem}
In the case $r=2$ the Strichartz estimates are not needed because one may only rely on the Sobolev embedding $H^1(\R^3)\hookrightarrow  L^6(\R^3)$. 
\end{rem}
Let us recall the main step in the proof of Proposition~\ref{LWP}. One may construct the solutions as the limit of the sequence $(u_n)_{n\geq 0}$, 
where $u_0=0$ and $u_{n+1}$ solves the linear problem 
\begin{equation} \label{eq:3.6}
 \partial_t^2 u_{n+1} - \Delta u_{n+1} + q(t, x) u_{n+1}+|u_n|^r u_n= 0,\: u(s, x) = f_1(x), \: \partial_t u(s, x) = f_2(x),
\end{equation}
where $t\in [s,s+\tau]$.  Set 
$$
\|u\|_{S}:=\|(u,\partial_t u)\|_{C([s,s+\tau];\hc)}+\|u\|_{L^{\frac{2r+2}{r-2}}_t([s, s+\tau], L^{2r+2}_x(\R^3))}\,.
$$
Using \eqref{eq:3.2} for $2 < r < 4$ with
\begin{equation} \label{eq:3.7}
\frac{1}{p} = \frac{r-2}{2r + 2},\quad \frac{1}{q} = \frac{1}{2r+2},
\end{equation}
 we obtain 
$$
\|u_{n+1}\|_{S}\leq C  \|(f_1,f_2)\|_{\hc}+C\|u_{n}\|^{r+1}_{L^{r+1}([s,s+\tau];L^{2r+2}_x(\R^3))}\,.
$$
Now using the H\"older inequality in time, we can write 
$$
\|u_{n}\|_{L^{r+1}([s,s+\tau];L^{2r+2}_x(\R^3))}
\leq \tau^{\frac{4-r}{2r + 2}} \|u_n\|_{L^{\frac{2r+2}{r-2}}_t([s, s+\tau], L^{2r+2}_x(\R^3))}
\leq 
 \tau^{\frac{4-r}{2r + 2}} \|u_n\|_{S}\,.
$$
Therefore, we arrive at the bound 
\begin{equation}\label{eq:3.8}
\|u_{n+1}\|_{S}\leq C  \|(f_1,f_2)\|_{\hc}+C\tau^{\frac{4-r}{2}}\|u_{n}\|^{r+1}_{S}\,.
\end{equation}
Assume that we have the estimate
$$\|u_{n}\|_{S}\leq  2C  \|(f_1,f_2)\|_{\hc}.$$
Applying  \eqref{eq:3.8}, and choosing $\tau$ so that
$$\tau^{\frac{4-r}{2}} (2C)^{r+1} \|(f_1, f_2)\|_{\hc}^r \leq 1,$$
we obtain the same bound for $\|u_{n+1}\|_{S}$. By recurrence we conclude that
$$
\|u_{n+1}\|_{S}\leq  2C  \|(f_1,f_2)\|_{\hc},\quad \forall\, n\geq 0.
$$
Next, let $w_n =u_{n+1} - u_n$ be a  solution of the problem
 $$\partial_t^2 w_n - \Delta w_n + q(t, x) w_n = |u_{n}|^r u_{n} - |u_{n + 1}|^r u_{n+1}, \,\, w_n(0, x) = \partial_t w_n(0, x) = 0.$$
By using the inequality
$$\Bigl||v|^r v- |w|^r w \Bigr|\leq  D_r |v-w|\Bigl(|v|^r + |w|^r\Bigr),$$
with constant $D_r$ depending only on $r$, we can similarly show that 
\begin{equation*}
\|u_{n+1}-u_n\|_{S}\leq \frac{1}{2}\|u_{n}-u_{n-1}\|_{S}
\end{equation*}
which implies the convergence of $(u_n)_{n \geq 0}$ with respect to the $\|\cdot\|_{S}$ norm.\\

Now assume that $(f_1, f_2) \in H^2(\R^3) \times H^1(\R^3)$ and introduce the norm
$$
\|u\|_{S_1}:=\|(u,\partial_t u)\|_{C([s,s+\tau];H^2(\R^3)\times H^1(\R^3))}+\|\nabla_x u\|_{L^{\frac{2r+2}{r-2}}_t([s, s+\tau], L^{2r+2}_x(\R^3))}\,.
$$
Therefore the sequence $(u_n)_{n \geq 0}$ satisfies the estimate

$$
\|u_{n+1}\|_{S_1}\leq C  \|(f_1,f_2)\|_{H^2(\R^3)\times H^1(\R^3)}+
C\||u_n|^r u_n\|_{L^1([s,s+a];H^1(\R^3))}
$$
and we have
\begin{equation*}
\||u_n|^r u_n\|_{L^1([s,s+a];H^1(\R^3))}
\leq 
C_r\tau^{\frac{4-r}{2}}\|u_{n}\|^{r}_{S}\|u_n\|_{S_1}.
\end{equation*}
which leads to
\begin{equation} \label{eq:3.100}
\|u_{n+1}\|_{S_1}\leq C_1  \|(f_1,f_2)\|_{H^2(\R^3) \times H^1(\R^3)}+C_1\tau^{\frac{4-r}{2}}\|u_{n}\|^{r}_{S}\|u\|_{S^1}\,.
\end{equation}
Indeed, we can write
$$
|u_n|^r u_n= u_n^{r/2+1}\overline{u_n}^{r/2}
$$
and therefore 
$$
\partial_{x_j}(u_n^{r/2+1}\overline{u_n}^{r/2})
=(r/2+1)\partial_{x_j} u_n u_n^{r/2}\overline{u_n}^{r/2}+r/2\,\,\overline{\partial_{x_j} u_n} u_n^{r/2+1}\overline{u_n}^{r/2-1}
$$
yields
$$
|\nabla_{x}(||u_n|^r u_n)|\leq C_r|\nabla_x u_n||u_n|^r.
$$
Applying  the H\"older inequality, one obtains
$$
\|\nabla_{x}(||u_n|^r u_n)|\|_{L^2_x}\leq 
C_1\|\nabla_x u_n\|_{L^{2r+2}_x(\R^3)}\||u_n|^r\|_{L^{\frac{2r+2}{r}}_x(\R^3)}
=C_1\|\nabla_x u_n\|_{L^{2r+2}_x(\R^3)} \|u_n\|^r_{L^{2r+2}_x(\R^3)}.
$$
Increasing, if it is necessary, the constant $C > 0$ we may arrange that (\ref{eq:3.8}) and (\ref{eq:3.100}) hold with the same constant. Therefore
we obtain a local solution $u(t, x) \in C([s , s + \tau], H^2(\R^3) \times H^1(\R^3))$ in the same interval $[s,s + \tau].$
\begin{rem}
We work in the complex setting, but if $(f_1,f_2)$ is real valued, then the solution remains real valued. 
Indeed, if $u$ is a solution of \eqref{eq:3.4} then so is $\overline{u}$ and we may apply the uniqueness to conclude that $u=\overline{u}$.  
\end{rem}
\subsection{Global well-posedness and polynomial bounds }
Fix $(f_1,f_2)\in \hc$.  Let $u$ be the local solution of \eqref{eq:3.4} obtained in Proposition~\ref{LWP} (with $s=0$).  
First we prove the following 
\begin{lem}\label{llcc}
The solutions 
$$u(t, x) \in  C([0, A], H^{2}(\R^3)) \cap C^1([0, A], H^1(\R^3)) \cap L_t^{\frac{2r +2}{r-2}}([0, A], L^{2r +2}_x(\R^3))$$
 of \eqref{eq:3.4} satisfy the relation 
\begin{equation}\label{eq:3.10}
\frac{d}{dt}\int_{\R^3} \big(\frac{1}{2}|\partial_t u|^2 + \frac{1}{2}|\nabla_x u|^2 + \frac{1}{2}q |u|^2 +\frac{1}{r+2} |u|^{r+2}\big)dx  = \frac{1}{2} \Re\int_{\R^3} (\partial_t q) |u|^2 dx,\: 0 \leq  t \leq A.
\end{equation}
\end{lem}
\begin{rem} We show that $(\ref{eq:3.10})$ holds in the sense of distributions ${\mathcal D}'(]0, A[)$. Since the right hand side of $(\ref{eq:3.10})$ is continuous in $]0, A[$ the derivative of the left hand side can be taken in the classical sense.
\end{rem}
\begin{proof} Let us  first remark  that $\int_{\R^3} |u|^{j+2}(t, x) dx \leq \|u(t, x)\|_{H^1_x(\R^3)}^{j +2}$ for $0\leq j<4$, thanks to the Sobolev embedding $H^1(\R^3)\hookrightarrow  L^{j+2}(\R^3)$.
Moreover, from our assumption it follows that $u(t, x) \in C([0, A], L^{\infty}_x(\R^3))$ and this implies
$$|u|^r(t, x) u (t,x) \in C([0, A], L^2_x(\R^3)).$$
Therefore, from the equation (\ref{eq:3.4}) we deduce $\partial_t^2 u(t, x) \in C([0, A],L^2_x(\R^3))$.

To verify (\ref{eq:3.10}), notice that
\begin{eqnarray*}
\Re \big(\int_{\R^3} (\partial_t^2 u - \Delta_x u + |u|^r u)\overline{\partial_t u} dx\big)  & = & - \Re\big(\int_{\R^3} q(t, x)u \overline{\partial_t u}dx\big)
\\
& = & -\frac{1}{2}\frac{d}{dt}\big(\int q |u|^2 dx\big) + \frac{1}{2}\int (\partial_t q) |u|^2 dx
\end{eqnarray*}
and the integrals 
$$\int_{\R^3} (\partial_t^2 u - \Delta_x u) \overline{\partial_t u} dx,\:\int_{\R^3}|u|^r u \bar{u}_t dx$$
are well defined. 
 After an approximation with smooth functions and integration by parts we deduce 
$$\Re \int_{\R^3} \Bigl(\partial_t^2 u -\Delta_x u\Bigr) \overline{\partial_t u} dx = \frac{d}{dt}\int_{\R^3} \frac{1}{2}(|\partial_t u|^2 + |\nabla_x u|^2) dx.$$
On the other hand,
$$
 (r/2 + 1)
(u^{\frac{r}{2}} \bar{u}^{\frac{r}{2} + 1} \partial_t u + u^{\frac{r}{2} + 1} \bar{u}^{\frac{r}{2}} \partial_t \bar{u})
 = \big(\partial_t (u^{\frac{r}{2} + 1}) \bar{u}^{\frac{r}{2} + 1} + \partial_t(\bar{u}^{\frac{r}{2} + 1})u^{\frac{r}{2} + 1}\big)$$
and hence 
$$ \Re \int_{\R^3}|u|^r u \bar{u}_t dx = \frac{1}{r+2}\frac{d}{dt}\big( \int_{\R^3}|u|^{r+2}dx\big).$$ 
Thus (\ref{eq:3.10}) holds for $0 < t < A$ and by continuity one covers the interval $[0, A].$ 
\end{proof}
We need the following simple lemma.
\begin{lem} \label{ode}
Let $0 < \gamma < 1$ and let $X(t) : [0,\infty) \rightarrow [0,\infty)$ be a derivable function such that for some $A>0$,
\begin{equation*} 
|X'(t)| \leq C X^{1- \gamma}(t), \quad 0 \leq t\leq A.
\end{equation*}
Then 
$$
X(t) \leq (X^{\gamma}(0) +C \gamma t)^{\frac{1}{\gamma}},\quad 0 \leq t \leq A.
$$
\end{lem}
\begin{proof}
First assume that $X(t) >0$ for all $0 \leq t \leq A.$ We have
$$\Bigl|\frac{d}{dt} (X^{\gamma}(t))\Bigr| = \gamma\Bigl|X^{\gamma - 1}(t)X'(t) \Bigr| \leq C\gamma.$$
Hence
$$X^{\gamma}(t) = \Bigl|\int_0^t (X^{\gamma})'(\tau) d\tau + X^{\gamma}(0)\Bigr| \leq X^{\gamma}(0) + C \gamma t$$
and we obtain the assertion for $X(t)>0$. 
In the general case,
we apply the previous argument to  $X(t) +\epsilon, \: \epsilon > 0$ and we let $\ep \to 0$. This completes the proof. 
\end{proof}
Let $u(t, x) \in C([0, A), H^2(\R^3) \cap C^1([0, A], H^1(\R^3))\cap L_t^{\frac{2r+2}{r-2}}([0, A], L^{2r + 2}_x(\R^3))$
be a solution of (\ref{eq:3.4}) and let
$$
X(t)=\int_{\R^3} \big(\frac{1}{2}|\partial_t u|^2 + \frac{1}{2}|\nabla_x u|^2 + \frac{1}{2}q |u|^2 +\frac{1}{r+2} |u|^{r+2}\big)dx \,.
$$
The support property $q(t, x) = 0$ for $|x|> \rho$ and the H\"older inequality  imply 
$$
\Big|\int_{\R^3} (\partial_t q) |u|^2 dx\Big| \leq C \|u(t,\cdot)\|^2_{L^2(|x|\leq \rho)} \leq C_1 \|u(t, \cdot))\|^{2}_{L^{r + 2}(|x| \leq \rho)}.
$$
Therefore
$$
|X'(t)| \leq C_2 X^{\frac{2}{r+2}}(t) = C_2 X^{1 - \frac{r}{r+2}}(t) 
$$
and  applying Lemma~\ref{ode},  we deduce
\begin{equation} \label{eq:3.11}
X(t) \leq \Bigl(X^{\frac{r}{r+2}}(0) +  \frac{C_2 r}{r+2}t\Bigr)^{\frac{r+2}{r}}\,\: 0 \leq t \leq A.
\end{equation}

As a consequence of \eqref{eq:3.11} we get 
\begin{equation*} 
\big(
\|\partial_t u(t, \cdot)\|^2_{L^2(\R^3)} + \|\nabla_x u(t, \cdot)\|^2_{L^2(\R^3)}
\big)^{\frac{1}{2}}
 \leq \sqrt{2}\Bigl(X^{\frac{r}{r+2}}(0) + \frac{C_2 r}{r+2} t\Bigr)^{\frac{r+2}{2r}}
\end{equation*}
and therefore 
\begin{equation*} 
\|\partial_t u(t, \cdot)\|_{L^2(\R^3)} + \|\nabla_x u(t, \cdot)\|_{L^2(\R^3)} \leq 2\Bigl(X^{\frac{r}{r+2}}(0) + \frac{C_2 r}{r+2} t\Bigr)^{\frac{r+2}{2r}}\,.
\end{equation*}
On the other hand,
$$X(0) \leq A_r\|(u,u_t)(0,x)\|_1^2\Bigl( 1 + \|(u, u_t)(0, x)\|_1^r\Bigr)$$
with a constant $A_r$ depending on $r.$ Hence from (\ref{eq:3.11}) we get 
$$\|\partial_t u(t, \cdot)\|_{L^2(\R^3)} + \|\nabla_x u(t, \cdot)\|_{L^2(\R^3)} \leq 2\Bigl(X^{\frac{r}{r+2}}(0) + \frac{C_2 r}{r+2} t\Bigr)^{\frac{r+2}{2r}}$$
\begin{equation} \label{eq:3.12}
\leq 2\Bigl(A_r^{\frac{r}{r+2}}\|(u,u_t)(0, x)\|_1^{\frac{2r}{r+2}}\Bigl[1 + \|(u,u_t)(0, x)\|_1^r\Bigr]^{\frac{r}{r+2}} + \frac{C_2 r}{r+2} t\Bigr)^{\frac{r+2}{2r}}, \: 0 \leq t \leq A.
\end{equation}
Finally, from
$$u(t, x) = u(0, x) + \int_0^t  \partial_t u(\tau, x) d\tau$$ 
one deduces
\begin{equation*} 
\|u(t, x)\|_{L^2} \leq \|u(0,x)\|_{L^2} + 2 t\Bigl(X^{\frac{r}{r+2}}(0) + \frac{C_2 r}{r+2}t\Bigr)^{\frac{r+2}{2r}}\,.
\end{equation*}
This yields a polynomial bound for the solutions 
$$u(t, x) \in C([0, A], H^2(\R^3))\cap C^1([0, A], H^1(\R^3)) \cap L_t^{\frac{2r+2}{r-2}}([0, A], L^{2r+2}_x(\R^3)).$$
 
Now we pass to the global existence of solution of (\ref{eq:3.4}). We will deal with the case $2< r < 4,$ while the case $r = 2$ can be covered by the Sobolev embedding theorem. We fix a number $a > 0$ and our purpose is to show that (\ref{eq:3.4}) has a solution for $t \in [0, a]$ with initial data $f\in \hc.$ We fix $p, q$ by (\ref{eq:3.7}) and let the Strichartz estimate
(\ref{eq:3.2}) holds in the interval $[0, a]$ with a constant $C_a > 0$. The above argument yields a local solution $u(t, x)$
with initial data $f = (f_1, f_2) \in \hc$ for $t \in [s, s +\tau]$. Recall that $\tau = c(1 + \|f\|_{\hc})^{-\gamma}.$ 
Introduce the number
$$B_a: = \|f\|_{\hc} + a(B_1 + B_2 a)^{\frac{r+2}{2r}},$$
where $B_1 > 0$ and $B_2 > 0$ depend only on $\|f\|_{\hc}$ and $r$. This number should be a bound of the energy of the solution $u(t, x)$ in $[0, a]$ with initial data $f \in \hc$ if the above argument based on Lemma~\ref{llcc} and Lemma~\ref{ode} works. However,  the proof of Lemma~\ref{llcc}  cannot be applied directly for functions $u(t, x) \in C([0, a], H^1(\R^3)) \cap C^1([0, a],L^2(\R^3)).$\\

Define $\tau(a):  = c(1 + B_a)^{-\gamma} < 1$ with the constants $c > 0, \gamma > 0$ of Proposition~\ref{LWP} and observe that the local existence theorem can be applied in the interval $[s, s + \tau(a)]\subset [0, a]$ if the norm of the initial data for $t = s$ is bounded by $B_a.$ To overcome the difficulty connected with Lemma~\ref{llcc} 
and since we did not prove in Proposition~\ref{LWP} the continuous dependence with respect to the initial data in $\hc$, we need to apply an approximation argument in 
$[s, s + \ep(a)]$, where the number $0 < \ep(a) \leq \tau(a)$ will be defined below. For simplicity we treat the case $s = 0$ below.
\\

By the local existence let $u(t, x)$ be the solution of (\ref{eq:3.4}) in $[0, \tau(a)]$ with initial data $f =(f_1, f_2)\in \hc.$ Choose a sequence $g_n = ((g_n)_1, (g_n)_2) \in H^2(\R^3) \times H^1(\R^3)$ converging in $\hc$ to $(f_1, f_2)\in \hc$ as $n \to \infty$ and let $w_n(t, x)$ be the solution of the problem (\ref{eq:3.4}) in the {\it same interval} $[0, \tau(a)]$ with initial data $g_n$. Then by Proposition~\ref{LWP},
$$w_n(t, x) \in C([0, \tau(a)], H^2(\R^3)\cap C^1([0, \tau(a)], H^1(\R^3)) \cap L_t^{\frac{2r+2}{r-2}}([0, \tau(a)], L^{2r + r}_x(\R^3)).$$ 
Set $v_n = w_n - u$. We claim that for $n \to \infty$ we have
\begin{equation*} 
\|(v_n, (v_n)_t)\|_{C([0, \ep(a)], \hc)} + \|v_n\|_{L^p_t([0, \ep(a)], L^q_x(\R^3))} \rightarrow 0
\end{equation*}
 with $0 <\ep(a) \leq \tau(a)$ defined below. Clearly, $v_n$ is a solution of the equation
$$\partial_t^2 v_n - \Delta v_n + q(t, x) v_n = |u|^r u - |w_n|^r w_n.$$
Applying (\ref{eq:3.2}), one obtains
\begin{eqnarray} \label{eq:3.13}
\|(v_n,(v_n)_t)\|_{C([0,\ep(a)], \hc)} + \|v_n\|_{L^{\frac{2r+2}{r-2}}_t([0, \ep(a)], L^{2r+2}_x(\R^3))}\nonumber \\
 \leq C_a \|g_n - f\|_{\hc} + C_a\||u|^r u - |w_n|^r w_n\|_{L^1([0,\ep(a)], L^2_x(\R^3))}
\end{eqnarray}
and
$$\|(|u|^r u- |w_n|^rw_n)(t, .)\|_{L^2_x} \leq  C\|v_n(t, .)\|_{L^{2r + 2}_x} \Bigl(\|u(t, .)\|_{L^{2r + 2}_x}^r + \|w_n(t, .)\|_{L^{2r + 2}_x}^r\Bigr).$$
Since $\frac{1}{p} + \frac{r}{p} + \Bigl(1 - \frac{r+1}{p}\Bigr) = 1$, by the generalized 
H\"older inequality in the integral with respect to $t$ in (\ref{eq:3.13}) for large $n \geq n_0$ we get
$$C_a\| |u|^r u- |w_n|^r w_n\|_{L^1([0,\ep(a)], L^2_x(\R^3))}$$
$$\leq D_r C_a\ep(a)^{(1 - \frac{r+1}{p})} \|v_n\|_{L^p([0,\ep(a)], L^q_x)} \Bigl(\|u\|^{r}_{L^p([0, \ep(a)], L^q_x)}+
\|w_n\|^{r}_{L^p([0, \ep(a)], L^q_x)}\Bigr)$$
$$ \leq 2D_r C_a^{r+1}(\|f\|_{\hc} + 1)^r\ep(a)^{(1 - \frac{r+1}{p})} \|v_n\|_{L^p([0,\ep(a)], L^q_x)}.$$
Here $D_r$ is a constant depending only on $r$ and we used that by Proposition~\ref{LWP} 
\begin{equation} \label{eq:3.14}
 \|w_n\|_{L^{\frac{2r+2}{r-2}}([0,\ep(a)], L^{2r + 2}_x)} \leq C_a \|g_n\|_{\hc} \leq C_a (\|f\|_{\hc} + 1),\: n \geq n_0
\end{equation}
with a similar estimate for $\|u\|_{L^{\frac{2r + 2}{r-2}}([0,\ep(a)], L^{2r+2}_x)}.$
Clearly, $1 - \frac{r+1}{p} = 2 - \frac{r}{2} > 0$ and we choose $0 < \ep(a) \leq \tau(a)$, so that
$$2 D_r C_a^{r + 1}(B_a + 1)^r \ep(a)^{(1 - \frac{r+1}{p})} \leq \frac{1}{2}.$$
Then we may absorb the term on right hand side of (\ref{eq:3.13}) involving $w_n, u$ and letting $n \to \infty$, we prove our claim. Moreover, for almost all $t \in [0, \ep(a)]$, taking into account (\ref{eq:3.14}),  we have
$$\Bigl|\int_{\R^3} \Bigl( |u(t, x)|^{r+2} - |w_n(t, x)|^{r+2}\Bigr) dx\Bigr|$$
$$  \leq D_r \|u(t, x) - w_n(t, x)\|_{L^2(\R^3)} \Bigl( \|u(t, x)\|^{r+1}_{L^{2r+2}_x(\R^3)} + \|w_n(t, x)\|^{r+1}_{L^{2r+2}_x(\R^3)}\Bigr) dx \longrightarrow_{n \to \infty} 0.$$
Consequently, we have
$$\int_{\R^3} \Bigl(\frac{1}{2}\Bigl(|\partial_t w_n|^2 + |\nabla_x w_n|^2 + q |u|^2\Bigr) +\frac{1}{r+2} |w_n|^{r+2}\Bigr)dx$$
$$ \longrightarrow_{n \to \infty} \int_{\R^3} \big(\frac{1}{2}(|\partial_t u|^2 + |\nabla_x u|^2 + q |u|^2) +\frac{1}{r+2} |u|^{r+2}\big)dx$$
 in the sense of distributions ${\mathcal D}'(0, \ep(a)).$ The equality (\ref{eq:3.10}) for $0 \leq t \leq \ep(a)$ holds for
 $w_n$ and passing to a limit in the sense of distributions, we conclude that (\ref{eq:3.10}) holds for $u(t, x)$ for $0 < t < \ep(a)$ and hence for $0 \leq t \leq \ep(a)$. The right hand side of (\ref{eq:3.11}) is continuous with respect to $t$, hence the derivative with respect to $t$ is taken in a classical sense. Thus we are in position to apply Lemma~\ref{ode} for the $u(t, x).$ Finally, we deduce (\ref{eq:3.12}) for the solution $u(t, x)$ and the norm $\|(u, u_t)(t, .)\|_{\hc}$
for $t \in [0, \ep(a)]$  is bounded by $B_a$ introduced above.\\

Now we pass to the second step in the interval $[\ep(a), 2\ep(a)] \subset [0, a]$. As it was mentioned above, we have a bound $B_a$ for the norm of the initial data $(u(\ep(a),x), u_t(\ep(a), x))$. By the local existence we have solution in $[\ep(a), 2\ep(a)]$ and $u(t, x)$ is defined in $[0, 2 \ep(a)].$ On the other hand, we may approximate the initial data  $(u(\ep(a),x), u_t(\ep(a), x))$ by functions $g_n^{(2)} \in H^2 \times H^1$ and by the above argument the solution $u(t, x)$ in $[\ep(a), 2\ep(a)]$ is approximated by solutions $w_n^{(2)}(t, x)$ for which (\ref{eq:3.11}) holds for $\ep(a) \leq t \leq 2 \ep(a).$ Thus (\ref{eq:3.10}) is satisfied for $u(t, x)$ for $\ep(a) \leq t < 2 \ep(a)$ and combining this with the first step, one concludes that the same is true for $0 \leq t \leq 2\ep(a).$ This makes possible to apply Lemma~\ref{ode} for $0 \leq t \leq 2 \ep(a)$ and to deduce (\ref{eq:3.12}) with uniform constants
leading to a bound by $B_a$.  We can iterate this procedure, since $\tau(a), \ep(a)$ depend only on $\|f\|_{\hc},C_a$ and $r$, while $B_a$ depends on $\|f\|_{\hc}, a$ and $r$. The solution $u(t, x)$ will be defined globally in a interval $[0, \alpha(a)]$ with $0 < a - \alpha(a) < \ep(a)$. Since $\alpha(a) > a - \ep(a) >a-1$ and $a$ is arbitrary, we have a global solution $u(t, x)$ defined for $t \geq 0.$ 
An application of Lemma 3 justifies the bound (\ref{eq:3.12}) for $u(t, x)$ and for all $t \geq 0$ with constants depending only on $\|f\|_{\hc}$ and $r$.
A similar analysis holds for negative times $t$.

\begin{rem}
It is likely that in the case $r=2$ by using the  approach of \cite{PTV} one may obtain polynomial bounds on the higher Sobolev norms $H^\sigma(\R^3) \times H^{\sigma-1}(\R^3)$, $\sigma>1,$ of the solutions of \eqref{eq:3.4}. 
\end{rem}
\subsection{A uniform bound}
As a byproduct of the (semi-linear) global well-posedness, we have the following uniform bound on the solutions of \eqref{eq:3.4}. 
\begin{prop}\label{frechet}
Let $R>0$ and $A>0$. Then there exists a constant $C(A, R) > 0$ such that for every $(f_1,f_2)\in \hc$ such that $\|(f_1,f_2)\|_{\hc}<R$ the solution $u(t, x)$ of \eqref{eq:3.4} satisfies 
\begin{equation} \label{eq:3.15}
\|u\|_{L^{r+1}([0,A];L^{2r+2}_x(\R^3))}\leq C(A, R)\|(f_1,f_2)\|_{\hc}\,.
\end{equation}
\end{prop}
\begin{proof}
Thanks to the global bounds on the solutions, we obtain that there exists $R'=R'(R,A)$ such that if $\|(f_1,f_2)\|_{\hc}<R$ then the corresponding solutions satisfies 
$$
\sup_{0\leq t\leq A}\|(u(t,\cdot),\partial_t u(t,\cdot)\|_{\hc}\leq R'\,.
$$
Denote by $\tau = \tau(A, R') > 0$  the local existence time for initial data having $\hc$ norm $\leq R'$, i.e. $\tau=c(1+R')^{-\gamma}$ with the notations of Proposition~\ref{LWP}. Next we split the interval $[0, A]$ in intervals of size $\tau.$ In every interval $[k\tau, (k+1)\tau]$ we apply the estimate (\ref{eq:3.2}) with $F = 0$ and constant $C_A$ independent on $k$. Thus we obtain a bound
$$\|u(t, x)\|_{L^{\frac{2r + 2}{r-2}}([k \tau, (k+1)\tau], L^{2r + 2}_x(\R^3))} \leq C_A^k \|(f_1, f_2)\|_{\hc}, \: 1 \leq k +1 \leq A/\tau.$$
By using the H\"older inequality for the integral with respect to $t$, we obtain easily (\ref{eq:3.15}).
\end{proof}
\section{Proof of Theorem~\ref{th3}}
Let $$\hc\ni f \rightarrow \uc(t, s)f = (v(t,x; s), v_t(t, x; s)) \in \hc$$
be the monodromy operator corresponding to the Cauchy problem (\ref{eq:3.3}) with initial data $f$ for $t = s$. For $\uc(t, s)$ we have the representation
\begin{equation} \label{eq:4.1}
\uc(t, s)f = V(t, s)f - \int_s^t V(t, \tau) Q_0\big(|\uc(\tau, s)f|^r\uc(\tau, s)f\big)d\tau,
\end{equation}
where
$$Q_0 = \begin{pmatrix} 0 & 0\\
1 & 0\end{pmatrix}.
$$
Therefore we can write $\uc(t + T, s + T)f$ as
$$
 V(t+ T, s + T)f- \int_{s+T}^{t + T} V(t + T,\tau) Q_0 
\big(|\uc(\tau, s+ T)f|^r  \uc(\tau, s+ T)f\big)
d\tau
$$
which in turn can be written as 
$$
V(t, s)f - \int_s^t V(t, \tau) Q_0 \big(|\uc(\tau + T, s + T)f|^r \uc(\tau + T, s + T)f \big)d\tau.
$$
By the uniqueness of the solution of the equation
$$\uc(t, s)f = V(t, s)f - \int_s^t V(t, \tau) Q_0(|\uc(\tau, s)f|^r  \uc(\tau, s)f )d\tau,$$
one deduces $\uc(t + T, s+ T) = \uc(t, s).$ Moreover, one has the property
$$\uc(p, r) = \uc(p, s) \circ \uc(s, r), \quad p, r, s \in \R.$$
For the solution $u(t, x; 0)$ of (\ref{eq:3.3}) (with $s=0$)  with initial data $f \in \hc$, set
$$
w_n = (u(nT, x; 0),  \partial_t u(nT, x; 0))= \uc(nT, 0)f,\: n \in \N.
$$
Therefore
\begin{equation} \label{eq:4.2}
w_{n+1} = \uc((n+1)T, 0)f = \uc((n+1)T, nT) \circ \uc(nT, 0)f = \uc(T, 0)w_n.
\end{equation}
Setting $\fc = \uc(T, 0)$, we obtain a system
\begin{equation} \label{eq:4.3}
w_{n+1} = \fc (w_n),\quad n \geq 0.
\end{equation}
with a {\it nonlinear map} $\fc: \hc \rightarrow \hc.$ Consider the linear map $L = V(T, 0): \hc \rightarrow \hc.$
Our purpose is to how that $L$ is the Fr\'echet derivative of $\fc$ at the origin in the Hilbert space $\hc$.
We use the representation
$$\fc(h) = Lh -\int_0^T V(T, \tau)Q_0\big(|u(\tau, x; h)|^r u(\tau, x;h)\big) d \tau,$$
where $u(t, x; h)$ is the solution of (\ref{eq:3.3}) with $s=0$ and initial data $h$ at time $0$. 
Using the Strichartz estimate and Proposition~\ref{frechet}, we obtain for $\|h\|_1\leq 1$ the bound 
$$
\sup_{0\leq t\leq T} \|\fc(h) -Lh\|_1\leq
 C \|u(t , x; h)\|^{r+1}_{L^{r+1}([0,T];L^{2r + 2}_x(\R^3)}\leq C  \|h\|_1^{r+ 1},
 $$
where $C > 0$ depends on $T$ but is independent of $h$. This implies immediately that $L$ is the Fr\'echet derivative of $\fc$ at the origin.
 
For the exponential instability at $u = 0$ we use following definition (see \cite{GTZ}). 
\begin{defn}
$(i)$ The equilibrium $u = 0$ is unstable if there exists $\eps > 0$ such that for every $\delta >0$ one can find a sequence $\{u_n\}$ of solution of $(\ref{eq:4.3})$ such that $0 < \|u_0\|_1 \leq \delta$ and $\|u_n\|_1\geq  \eps$ for some $n \in \N.$\\
$(ii)$ The equilibrium $u = 0$ is exponentially unstable at rate $\rho > 1$ if there exist $\epsilon > 0$ and $C > 0$ such that
for every $\delta > 0$ one can find a sequence $\{u_n\}$ of solution of  $(\ref{eq:4.3})$  satisfying $0 < \|u_0\|_1 \leq \delta$
and $\|u_N\|_1 \geq C \rho^N \|u_0\|_1$ for any $N$ for which we have 
$$\max\{\|u_0\|_1,...,\|u_N\|_1\} \leq \epsilon.$$
\end{defn} 

Clearly, the exponential instability implies instability. We consider the case when
the spectral radius $r(L)$ of $L$ is greater than 1. The analysis in Section~2 shows that there exist positive potentials $q(t, x) \geq 0$ for which $r(L) > 1$. We will apply the Rutman-Dalecki theorem or a more general version due to D. Henry (Theorem 5.1.5 in \cite{He}). For this purpose we need the condition
\begin{equation} \label{eq:4.4}
\|\fc(u) - L u\|_1 \leq b \|u\|_1^{1 + p}\: {\rm whenever}\:\|u\|_1 \leq a
\end{equation}
for some $a > 0,\: b > 0$ and $p > 0.$ In our case the condition (\ref{eq:4.4}) holds with $p = r$ and $a =1$.  Thus we obtain the following
\begin{thm}\label{th4} 
Assume that the linear operator $L$ has spectral radius $r(L) > 1$. Then $\fc$ is exponentially unstable  at $u = 0$ with rate $r(L).$
\end{thm}
It remains to observe that Theorem~\ref{th4} implies Theorem~\ref{th3}. 
\begin{rem}
The above argument showing nonlinear instability crucially relies on the fact that we deal with a semi-linear problem,  i.e. the solution map of \eqref{eq:3.4} is of class $C^1$ on  $\hc$.
It is worth to mention that there are examples of problems which are not semi-linear (the solution map is  not of class $C^1$) for which one can still get 
the nonlinear instability of some particular solutions (known to be linearly unstable).  In such cases a "more nonlinear approach" is needed. We refer to \cite{Grenier,RT} for more details on this issue. 
\end{rem}
\section{Generalizations}
 
We can consider more general nonlinear equations

\begin{equation} \label{eq:5.1}
\partial_t^2 u - \Delta_x u + |u|^r u +\sum_{j = 0}^{r-1} q_j(t, x)|u|^j u = 0, \:r = 2,3
\end{equation}
with time-periodic functions $q_j(t + T_j, x) = q_j(t, x) \geq 0,\: j = 0,\cdots,  r-1$ having compact support with respect to $x$.
For solutions
 $$u(t, x) \in C([0, \tau], H^2(\R^3)) \cap C^1([0,\tau], H^1(\R^3))\cap L^{\frac{2r+2}{r-2}}_t([0, A], L^{2r+2}_x(\R^3))$$ we obtain
$$\Re \Bigl(\int_{\R^3} (\partial_t^2 u - \Delta_x u + |u|^r u)\bar{u}_tdx \Bigr) = - \Re(\int_{\R^3} \sum_{j=0}^{r-1}q_j(t, x)|u|^ju \bar{u}_tdx)$$
$$= -\frac{d}{dt}\sum_{j=0}^{r-1}\Bigl(\int_{\R^3} \frac{1}{j+ 2}q_j |u|^{j + 2} dx\Bigr) + \sum_{j=0}^{r-1}\frac{1}{j + 2}\int_{\R^3} (q_j)_t |u|^{j +2} dx
$$
and
$$\frac{1}{j + 2}\Bigl|\int_{\R^3} (q_j)_t |u|^{j +2} dx\Bigr| \leq C_j \Bigl(\int_{\R^3} |u|^{r+2} dx \Bigr) ^{1 - \frac{r- j}{r+2}}, \: j= 0,\cdots ,r-1. $$
Setting
$$X(t) \equiv \int_{\R^3}\Bigl(\frac{1}{2}|u_t|^2(t, x) +\frac{1}{2}|\nabla_x u|^2(t, x) +\sum_{j=0}^{r-1} \frac{1}{j+2}q_j |u|^{j +2}(t, x) +\frac{1}{r+2}|u|^{r+2}(t, x)\Bigr) dx, \: 0\leq t \leq A,$$
one deduce
$$|X'(t)| \leq B_r \sum_{j=0}^{r-1}X(t)^{1 - \frac{r-j}{r+2}}
\leq B_r(1+X(t))^{1-\frac{1}{r+2}}\,.
$$
Therefore we can apply  Lemma~\ref{ode} to the quantity $Y(t)=1+X(t)$ which implies, as before, the global existence and  the polynomial bounds for the  Cauchy problem for (\ref{eq:5.1}). 

\end{document}